\documentclass[leqno,11pt]{amsart}
\usepackage{}
\usepackage{amssymb, amsmath}
\usepackage{amsthm, amsfonts,mathrsfs}
\def\inte#1{
\displaystyle\mathop{#1\kern0pt}^\circ }

\def\virgp{\raise 2pt\hbox{,}}
\def\cdotpv{\raise 2pt\hbox{;}}

\def\C{\mathop{\mathbb C\kern 0pt}\nolimits}
\def\DD{\mathop{\mathbb D\kern 0pt}\nolimits}
\def\EE{\mathop{{\mathbb E \kern 0pt}}\nolimits}
\def\K{\mathop{\mathbb K\kern 0pt}\nolimits}
\def\N{\mathop{\mathbb N\kern 0pt}\nolimits}
\def\Q{\mathop{\mathbb Q\kern 0pt}\nolimits}
\def\R{\mathop{\mathbb R\kern 0pt}\nolimits}
\def\SS{\mathop{\mathbb S\kern 0pt}\nolimits}
\def\ZZ{\mathop{\mathbb Z\kern 0pt}\nolimits}
\def\TT{\mathop{\mathbb T\kern 0pt}\nolimits}
\def\P{\mathop{\mathbb P\kern 0pt}\nolimits}

\newcommand{\beq}{\begin{equation}}
\newcommand{\eeq}{\end{equation}}
\newcommand{\ben}{\begin{eqnarray}}
\newcommand{\een}{\end{eqnarray}}
\newcommand{\beno}{\begin{eqnarray*}}
\newcommand{\eeno}{\end{eqnarray*}}

\newtheorem{lem}{Lemma}[section]

\newtheorem{prop}{Proposition}[section]
\renewcommand{\theequation}{\thesection.\arabic{equation}}

\begin{document}

\title[Liouville type theorems for the stationary Hall-MHD equations]
{Liouville type theorems for the stationary Hall-MHD equations in local Morrey spaces}

\author[Z. Li and  Y. Su]{Zhouyu Li\textsuperscript{1}      \and   Yifan Su\textsuperscript{2}}

\thanks{$^1$ School of Sciences, Xi'an University of Technology, Xi'an 710054, China}
\thanks{$^2$ School of Mathematics, Northwest University, Xi'an 710127, China}
\thanks{E-mail address: zylimath@163.com (Z. Li); mayifansu@163.com (Y. Su)}

\begin{abstract}
This paper is concerned with the Liouville type theorems for the 3D stationary incompressible Hall-MHD equations. We establish that under some sufficient conditions in local Morrey spaces, solutions of the stationary Hall-MHD equations are identically zero. In particular, we also prove Liouville type results for the stationary incompressible MHD equations  on $\mathbb{R}^3$. Our theorems extend and generalize the classical results for the stationary incompressible Navier-Stokes equations.

\end{abstract}


\date{}

\maketitle



\vskip 0.2cm


\renewcommand{\theequation}{\thesection.\arabic{equation}}
\setcounter{equation}{0}
\section{Introduction}
We are concerned in this paper with the following stationary incompressible Hall-MHD equations on the whole space $\mathbb{R}^3$:
\begin{equation}\label{HMHD}
    \begin{cases}
    - \Delta \mathbf{u} + \operatorname{div}(\mathbf{u}\otimes \mathbf{u})+\nabla \Pi -\operatorname{curl} \mathbf{B}\times\mathbf{B}=0, \\
     -\Delta\mathbf{B} + \operatorname{curl}(\mathbf{B}\times \mathbf{u})+ \operatorname{curl}(\operatorname{curl}\mathbf{B}\times\mathbf{B})=0, \\
     \operatorname{div} \mathbf{u}=\operatorname{div} \mathbf{B} = 0.
    \end{cases}
\end{equation}
Here $x=(x_1, x_2, x_3)\in \mathbb{R}^3$ is the space variable. The unknown functions $\mathbf{u}$, $\mathbf{B}$ and $\Pi$ denote the velocity field, the magnetic field and the pressure, respectively.
The Hall term $\operatorname{curl}(\operatorname{curl}\mathbf{B}\times\mathbf{B})$ is derived from the Ohm's law and describes deviation from charge neutrality between the electrons and the ions. Hall term plays an
important role in understanding the magnetic reconnection phenomena in plasmas, neutron stars, star
formation, etc. For more physical explanations, see \cite{A2011, F1991, H2005, L1960}.

The main purpose of this paper is to study Liouville type problem for the system \eqref{HMHD}, which is inspired by the development of Navier-Stokes equations. If $\mathbf{B}=0$, the system \eqref{HMHD} reduces to the stationary incompressible Navier-Stokes equations. A very challenging open question is when the Dirichlet integral
$$\int_{\mathbb{R}^3}|\nabla \mathbf{u}|^2 \, dx < +\infty,$$
whether there exist a nontrivial solution. This uniqueness problem, or equivalently the Liouville
type problem originated from Leray's paper \cite{L1933}, and has been studied extensively recently. The most well-known one is due to Galdi (see \cite{G2011}, Remark X. 9.4, pp.729), which proved that if the smooth solution $\mathbf{u}$ satisfies
$$\mathbf{u}\in L^{\frac{9}{2}}(\mathbb{R}^3),$$
then $\mathbf{u}=0$. The key to the proof is essentially based on the following local estimate (also called the Caccioppoli type inequality):
\begin{equation*}
     \begin{split}
     \int_{|x|\leq R}|\nabla \mathbf{u}|^2 \, dx &\leq C\|\mathbf{u}\|^3_{L^\frac{9}{2}(R\leq |x|\leq 2R)}+ CR^{-\frac{1}{3}}\|\mathbf{u}\|^2_{L^\frac{9}{2}(R\leq |x|\leq 2R)}\\
     &\qquad+ C\|\mathbf{u}\|_{L^\frac{9}{2}(R\leq |x|\leq 2R)}\|\Pi\|_{L^\frac{9}{4}(R\leq |x|\leq 2R)}
     \end{split}
\end{equation*}
for all $R>1$ and a constant $C$ independent of $R$. Later, Chae and Wolf in \cite{C2016} obtained a logarithmic improvement to Galdi's result under the assumption $$\int_{\mathbb{R}^3}|\mathbf{u}|^{\frac{9}{2}}\left\{\operatorname{log}\left(2+\frac{1}{|\mathbf{u}|}\right)\right\}^{-1}\, dx<\infty.$$
Seregin in \cite{S2016} established that $\mathbf{u}\in L^6(\mathbb{R}^3)\cap BMO^{-1}(\mathbb{R}^{3}) $  leads to the Liouville type theorem.
The result of Galdi was also extended to the setting of Lorentz space $L^{\frac{9}{2}, \infty}(\mathbb{R}^3)$ by Kozono et al. in \cite{KTW2017} and a certain kind of local Lorentz type space by Seregin and Wang in \cite{Wang2018}.
Moreover, the authors in \cite{Cha2021, Ja2020, Ja2021} proved the Liouville type results in a more general setting of Morrey spaces and some kind of local Morrey spaces.

However, the Liouville type problem for MHD type system is quite different due to the lack of maximum principle. For the case of MHD equations, Chae and Weng in \cite{CW2016} proved that
the smooth solution $ (\mathbf{u}, \mathbf{B})=0 $ under the condition
\begin{equation*}
     \begin{split}
(\mathbf{u}, \mathbf{B}) \in L^{3}(\mathbb{R}^{3})\quad
\mathrm{and} \quad (\nabla\mathbf{u}, \nabla\mathbf{B})\in L^2(\mathbb{R}^{3}).
     \end{split}
\end{equation*}
Schulz in \cite{Sc2018} showed if $ (\mathbf{u}, \mathbf{B}) \in L^{6}(\mathbb{R}^{3}) \cap BMO^{-1}(\mathbb{R}^{3}) ,$
then  $ (\mathbf{u}, \mathbf{B})=0 $. For the case of Hall-MHD equations, in the famous paper \cite{C2014}, Chae, et al. proved that if
$(\mathbf{u}, \mathbf{B})$ is a smooth solution of \eqref{HMHD} such that
\begin{equation*}
     \begin{split}
     \int_{\mathbb{R}^3}|\nabla \mathbf{u}|^2+ |\nabla \mathbf{B}|^2\, dx<\infty \quad
\mathrm{and} \quad (\mathbf{u}, \mathbf{B})\in L^{\infty}(\mathbb{R}^3)\cap L^{\frac{9}{2}}(\mathbb{R}^3),
     \end{split}
\end{equation*}
then $(\mathbf{u}, \mathbf{B})=0$. Recently,  the first author and Niu in \cite{Li2020} extend the result of Chae, et al. in \cite{C2014}
to the Lorentz space $L^{\frac{9}{2}, \infty}(\mathbb{R}^3)$. For more interesting Liouville type results about MHD and
Hall-MHD equations, we refer to \cite{Ch2021, F2021, LP2021, L2020, Y2020} and the references therein.

Here we consider the Liouville type theorems for the stationary incompressible Hall-MHD equations \eqref{HMHD}. Motivated by \cite{Cha2021, Ja2021, KTW2017}, we investigate some Liouville type theorems provided the velocity and magnetic field belongs to a fairly general functional setting, that is, some local Morrey spaces. It is a natural way to extend the space widely and improve the previous results.

An outline of this paper is as follows: we state our main results in Section 2. In Section 3, we collect some elementary facts, which will be used later. In Section 4, we obtain a characterization
of the pressure term for the Hall-MHD equations. Section 5 is devoted to obtaining the Caccioppoli type inequalities, which are essential for achieving our results.
In Section 6, the proofs of Theorems 1 and 2 are given.


\renewcommand{\theequation}{\thesection.\arabic{equation}}
\setcounter{equation}{0} 

\section{Statement of the results}
Before stating the main theorems, we first recall the definitions of Morrey spaces and local Morrey spaces. For $1<p<r<+\infty$, the homogeneous Morrey space $\dot{M}^{p, r}(\mathbb{R}^3)$ is the set of all $f\in L^p_{loc}(\mathbb{R}^3)$ such that
\begin{equation}\label{def-1}
     \begin{split}
    \|f\|_{\dot{M}^{p, r}}=\sup_{R>0, x_0\in \mathbb{R}^3} R^{\frac{3}{r}}\left(R^{-3}\int_{B(x_0, R)}|f(x)|^p \, dx\right)^{\frac{1}{p}}< +\infty ,
     \end{split}
\end{equation}
where $B(x_0, R)$ is a ball of radius $R$ centered at $x_0$. Here the term $R^{\frac{3}{r}}$ characterizes the decaying of the averaged quantity $\left(R^{-3}\int_{B(x_0, R)}|f(x)|^p \, dx\right)^{\frac{1}{p}}$
when $R$ is large. The spaces satisfy the following embedding relation
\begin{equation*}
     \begin{split}
 L^r(\mathbb{R}^3)\hookrightarrow L^{r,q}(\mathbb{R}^3)\hookrightarrow\dot{M}^{p, r}, \quad 1<p<r\leq q\leq +\infty,
     \end{split}
\end{equation*}
where the space $L^{r,q}(\mathbb{R}^3)$ is the Lorentz space.

The local Morrey spaces describe the averaged decaying of functions in a more general setting. Let $\gamma\geq0$ and $1<p<+\infty$ , the local Morrey space $M^p_\gamma(\mathbb{R}^3)$ is the set of functions
$f\in L^p_{loc}(\mathbb{R}^3)$ such that
\begin{equation}\label{def-2}
     \begin{split}
   \|f\|_{M^p_\gamma}=\sup_{R\geq 1} \left( R^{-\gamma} \int_{B(0, R)} |f(x)|^p \, dx \right)^{\frac{1}{p}}< +\infty .
     \end{split}
\end{equation}
It is easy to verify that $M^p_\gamma(\mathbb{R}^3)$ is a Banach space under the norm $\|\cdot\|_{M^p_\gamma}$. Remark that the parameter $\gamma$ describes the behavior of the quantity
$\left( \int_{B(0, R)} |f(x)|^p \, dx \right)^{\frac{1}{p}}$ when $R$ is large. Furthermore, the spaces satisfy the continuous embedding
\begin{equation*}
     \begin{split}
  M^p_{\gamma_1}(\mathbb{R}^3)\hookrightarrow M^p_{\gamma_2}(\mathbb{R}^3), \quad \gamma_1\leq\gamma_2.
     \end{split}
\end{equation*}
We emphasize that for $1<p<r<+\infty$, taking the parameter $\gamma$ such that $3(1-p/r) <\gamma$, then
\begin{equation*}
     \begin{split}
 \dot{M}^{p, r}(\mathbb{R}^3)= M^p_{3(1-p/r)}(\mathbb{R}^3)\hookrightarrow M^p_{\gamma}(\mathbb{R}^3).
     \end{split}
\end{equation*}
Hence, in this sense local Morrey space $M^p_\gamma(\mathbb{R}^3)$ can be regarded as a generalization of homogeneous Morrey space $\dot{M}^{p, r}(\mathbb{R}^3)$. For more references about the Morrey spaces and the
local Morrey spaces, see \cite{F2020, L2016}.

Finally, we define the space $M^p_{\gamma, 0}(\mathbb{R}^3)$ as the subspace of all functions $f\in M^p_\gamma(\mathbb{R}^3)$ such that
\begin{equation*}
     \begin{split}
    \lim_{R\rightarrow +\infty} \left( R^{-\gamma} \int_{\frac{R}{2}\leq|x|\leq R} |f(x)|^p \, dx  \right)^{\frac{1}{p}} =0.
     \end{split}
\end{equation*}

To simplify the notations, we denote
$$\eta=\eta(p, \gamma)= \frac{\gamma}{p}+\frac{2}{3}-\frac{3}{p},$$
$B_R=\{x\in \mathbb{R}^3: |x|\leq R \}$ and $C_R=\{x\in \mathbb{R}^3: R\leq |x|\leq 2R \}$.
We state the first result of this paper as follows:

\noindent\textbf{Theorem 1.} \label{thm-main-1}
Suppose that $(\mathbf{u}, \mathbf{B})$ is a smooth solution to \eqref{HMHD} with $\nabla \mathbf{B} \in L^2(\mathbb{R}^3)$ and $(\mathbf{u}, \mathbf{B})\in M^p_\gamma(\mathbb{R}^3)$ for $0<\gamma<3\leq p<+\infty$.\\
{\rm{(i)}} If $\eta(p, \gamma)<0$, then $\mathbf{u}=\mathbf{B}= 0$ on $\mathbb{R}^3$.\\
{\rm{(ii)}} If $\eta(p, \gamma)=0$, and if moreover
\begin{equation}\label{con-0}
     \begin{split}
   \|\mathbf{u}\|^3_{{M^p_\gamma}(\mathbb{R}^3)}+\|\mathbf{u}\|_{{M^p_\gamma}(\mathbb{R}^3)}\|\mathbf{B}\|^2_{{M^p_\gamma}(\mathbb{R}^3)}
   \leq \delta \|(\nabla \mathbf{u}, \nabla \mathbf{B})\|^2_{L^2(\mathbb{R}^3)}
     \end{split}
\end{equation}
for some $0< \delta < 1/ C_0$, then $\mathbf{u}=\mathbf{B}= 0$ on $\mathbb{R}^3$.\\
{\rm{(iii)}} If $\eta(p, \gamma)>0$, and if moreover
\begin{equation}\label{con-1}
     \begin{split}
    \lim_{R\rightarrow +\infty} R^{3\eta(p, \gamma)}\left((R^{-\gamma}\int_{C_R}|\mathbf{u}|^p \, dx)^{\frac{1}{p}} + (R^{-\gamma}\int_{C_R}|\mathbf{B}|^p \, dx)^{\frac{1}{p}}\right) =0,
     \end{split}
\end{equation}
then $\mathbf{u}=\mathbf{B}= 0$ on $\mathbb{R}^3$.

On the other hand, in the framework of a slightly smaller space $M^p_{\gamma, 0}(\mathbb{R}^3)$, we do not need any additional assumption similar to \eqref{con-0}.
Specifically, this idea can be stated as follows.

\noindent\textbf{Theorem 2.}\label{thm-main-2}
Suppose that $(\mathbf{u}, \mathbf{B})$ is a smooth solution to \eqref{HMHD} with $\nabla \mathbf{B} \in L^2(\mathbb{R}^3)$, $\mathbf{u}\in M^p_{\gamma, 0}(\mathbb{R}^3)$ and $\mathbf{B}\in M^p_\gamma(\mathbb{R}^3)$ for $0<\gamma<3\leq p<+\infty$. If $(p, \gamma)$ is such that $\eta(p, \gamma)= 0$, then $\mathbf{u}=\mathbf{B}= 0$ on $\mathbb{R}^3$.

It should be noted that in Theorems 1 and 2 we always set the parameters $0<\gamma<3\leq p<+\infty$. Since our methods deeply rely on the Cacciopoli type inequalities, for more details see Propositions \ref{prop-1}
and \ref{prop-2} below, we need the condition $3\leq p<+\infty$. In addition, in order to obtain the estimate for the pressure term,  we need to use the $\operatorname{Calder\acute{o}n}$-Zygmund inequality in the spaces $ M^p_\gamma(\mathbb{R}^3)$, that is, point {\rm{(i)}} of Lemma \ref{Lem-1}, so the condition $0<\gamma<3$ is required.

The main interest of our results is based on the fact that we use a fairly general framework to solve the Liouville
type problem.
Indeed, we know that if $\eta(p, \gamma)\leq 0$, then for $3<r<\frac{9}{2}$ and $3\leq p< r$, we have the embeddings
\begin{equation*}
     \begin{split}
    L^r(\mathbb{R}^3)\hookrightarrow L^{r, \infty}(\mathbb{R}^3)\hookrightarrow \dot{M}^{p, r}(\mathbb{R}^3)\hookrightarrow M^p_\delta(\mathbb{R}^3)\hookrightarrow M^p_{\gamma, 0}(\mathbb{R}^3),
     \end{split}
\end{equation*}
involving the Lebesgue, Lorentz, Morrey and local Morrey spaces. Here, the parameter $\delta$ satisfies $3(1-p/r)<\delta<\gamma$.
In particular, for the values $\gamma=1$ and $p=3$, and for
$r=\frac{9}{2}$ and $\frac{9}{2}< q <+\infty$, we also have the embeddings
\begin{equation*}\label{emb}
     \begin{split}
    L^{\frac{9}{2}}(\mathbb{R}^3)\hookrightarrow L^{\frac{9}{2}, q}(\mathbb{R}^3)\hookrightarrow M^3_{1, 0}(\mathbb{R}^3).
     \end{split}
\end{equation*}
Notice that when $\gamma=1$ and $p=3$, we  have $\eta(3, 1)=0$. For more references about the inclusion relationship of local Morrey spaces, we recommend \cite{Osc2021, Ja2021}.
With the above facts in hand, we now give several remarks.

{\rm{(i).}}
Our results improve the result of  Chae, et al. in \cite{C2014}, and the result of \cite{Li2020} is extended to the local Morrey spaces.

{\rm{(ii).}}When removing the Hall term $\operatorname{curl}(\operatorname{curl}\mathbf{B}\times\mathbf{B})$, the system \eqref{HMHD} reduces to the stationary incompressible MHD system. In this case, we do not
need the assumption condition $\nabla \mathbf{B}\in L^2(\mathbb{R}^3)$, because of the fact that this condition is only used in dealing with the Hall term, for more details see \eqref{3-8} below. Hence we also obtain
two new results for incompressible MHD system, which generalize the related result given in \cite{L2020} to the local Morrey spaces.

{\rm{(iii).}}When $\mathbf{B}=0$, the system \eqref{HMHD} reduces to the stationary incompressible Navier-Stokes
system. Clearly, our results generalize the corresponding results of Galdi \cite{G2011} and Jarr\'{i}n \cite{Ja2021}.


\renewcommand{\theequation}{\thesection.\arabic{equation}}
\setcounter{equation}{0} 

\section{Preliminaries}
For the convenience of readers, we will summarize some elementary facts about the local Morrey spaces and the weighted Lebesgue spaces.

Let us recall the definition of the weighted Lebesgue spaces.
For $\gamma \geq 0$ and $1< p< +\infty$, we consider the weight $\omega_\gamma(x)=\frac{1}{(1+|x|)^\gamma}$, then the weighted Lebesgue spaces $L^p_{\omega_\gamma}(\mathbb{R}^3)$ is defined as $L^p_{\omega_\gamma}(\mathbb{R}^3)=L^p(\omega_\gamma\, dx).$ It should be stressed that by using the Cauchy-Schwarz inequality we have the
following continuous embedding:
\begin{equation}\label{em}
     \begin{split}
L^p_{\omega_\gamma}(\mathbb{R}^3)\hookrightarrow L^\frac{p}{2}_{\omega_\gamma}(\mathbb{R}^3)
     \end{split}
\end{equation}
for $ \frac{9}{4}< \gamma< 3$.

In order to study the characterization of the pressure term for the general Hall-MHD equations, we introduce a time-space version of the local Morrey spaces. Let $\gamma>0$ and $1< p< +\infty$. We define
the space $M^p_\gamma L^p(0, T)$ as the set of functions $f\in L^p_{loc}([0, T]\times \mathbb{R}^3)$ such that
\begin{equation*}
     \begin{split}
   \|f\|_{M^p_\gamma L^p(0, T)}=\sup_{R\geq 1} \left( R^{-\gamma} \int_0^T \int_{B(0, R)} |f(x, t)|^p \, dxdt \right)^{\frac{1}{p}}< +\infty.
     \end{split}
\end{equation*}
Furthermore, we define the space $M^p_{\gamma, 0} L^p(0, T)$ as the set of functions $f \in M^p_\gamma L^p(0, T) $ such that
\begin{equation*}
     \begin{split}
    \lim_{R\rightarrow +\infty} \left( R^{-\gamma} \int_0^T \int_{\frac{R}{2}\leq|x|\leq R} |f(x)|^p \, dxdt \right)^{\frac{1}{p}}  =0.
     \end{split}
\end{equation*}
For details see Section 7 of \cite{F2020}. Finally, we give two useful lemmas.
\begin{lem}[Lemma 2.1 of \cite{Osc2021}]\label{Lem-1a}
Assume $0\leq \gamma < \delta$, $1< p< +\infty$ and $0<T<+\infty$. Then the following inclusion relationships are true:\\
{\rm{(i)}} $L^p_{\omega_\gamma}(\mathbb{R}^3)\hookrightarrow M^p_{\gamma, 0}(\mathbb{R}^3)\hookrightarrow M^p_\gamma(\mathbb{R}^3) \hookrightarrow L^p_{\omega_\delta}(\mathbb{R}^3)$;\\
{\rm{(ii)}} $ L^p([0, T]; L^p_{\omega_\gamma}(\mathbb{R}^3))\hookrightarrow M^p_{\gamma, 0}L^p(0, T) \hookrightarrow M^p_\gamma L^p(0, T) \hookrightarrow L^p([0, T]; L^p_{\omega_\delta}(\mathbb{R}^3)).$
\end{lem}

\begin{lem}[Lemma 2.2 and Corollary 2.1 of \cite{Osc2021}]\label{Lem-1}
Assume $0< \gamma <3$ and $1<p<+\infty$. Then there hold\\
{\rm{(i)}} $\|\nabla^2(-\Delta)^{-1}f\|_{M^p_\gamma(\mathbb{R}^3)}\leq C(p, \gamma)\|f\|_{M^p_\gamma(\mathbb{R}^3)}$;\\
{\rm{(ii)}} $\|\mathcal{M}_f\|_{M^p_\gamma(\mathbb{R}^3)}\leq C(p, \gamma)\|f\|_{M^p_\gamma(\mathbb{R}^3)}$, where $\mathcal{M}$ is the Hardy-Littlewood maximal function operator;\\
{\rm{(iii)}} The points {\rm{(i)}} and {\rm{(ii)}} also hold for the weighted Lebesgue spaces $L^p_{\omega_\gamma}(\mathbb{R}^3).$
\end{lem}


\renewcommand{\theequation}{\thesection.\arabic{equation}}
\setcounter{equation}{0} 

\section{characterization of the pressure term}
In this section, we aim at establishing a characterization of the pressure term for the following time-dependent Hall-MHD equations:
\begin{equation}\label{THMHD}
    \begin{cases}
    \partial_t \mathbf{u} - \Delta\mathbf{u} + \operatorname{div}(\mathbf{u}\otimes \mathbf{u})+\nabla \Pi -\operatorname{curl} \mathbf{B}\times\mathbf{B}=0, \quad (t, x)\in \mathbb{R}^+\times \mathbb{R}^3, \\
    \partial_t \mathbf{B} -\Delta\mathbf{B} + \operatorname{curl}(\mathbf{B}\times \mathbf{u})+ \operatorname{curl}(\operatorname{curl}\mathbf{B}\times\mathbf{B})=0,  \quad (t, x)\in \mathbb{R}^+\times \mathbb{R}^3, \\
    \operatorname{div} \mathbf{u}=\operatorname{div} \mathbf{B} = 0, \quad (t, x)\in \mathbb{R}^+\times \mathbb{R}^3,\\
    (u, B)|_{t=0} = (u_0, B_0), \quad x\in \mathbb{R}^3.
    \end{cases}
\end{equation}
More precisely, we show that the pressure $\Pi$ is always related to the velocity  field $\mathbf{u}$ and the magnetic field $\mathbf{B}$ in the general setting of the time-space local Morrey spaces. It should be noted that the
following proposition has an independent interest when seeking for very general frameworks in which the pressure is related to the other unknowns in the system \eqref{THMHD}.
\begin{prop} \label{prop-1a}
Assume that $(\mathbf{u}, \mathbf{B}, \Pi)$ is a solution of the system \eqref{THMHD} with $\Pi\in \mathcal{D}'([0, T]\times \mathbb{R}^3)$ and $(\mathbf{u}, \mathbf{B})\in M^p_\gamma L^p(0, T)$ for $0< \gamma<3$, $2< p< +\infty$.
Then the term $\nabla \Pi$ is given by the formula
\begin{equation*}
  \nabla \Pi=\nabla \left( \sum_{i,j=1}^3(-\Delta)^{-1}\partial_i\partial_j(u_iu_j-B_iB_j) -\frac{|\mathbf{B}|^2}{2} \right).
\end{equation*}
\end{prop}

\begin{proof}
Motivated by the method in \cite{Ja2021}, we first define $P$ given by the expression
\begin{equation}\label{P}
  P= \sum_{i,j=1}^3(-\Delta)^{-1}\partial_i\partial_j(u_iu_j-B_iB_j) -\frac{|\mathbf{B}|^2}{2}.
\end{equation}
Noting $(\mathbf{u}, \mathbf{B})\in M^p_\gamma L^p(0, T)$ for $0< \gamma< \frac{3}{2}$, we apply point ${\rm{(ii)}}$ of Lemma \ref{Lem-1a} to get $(\mathbf{u}, \mathbf{B})\in L^p([0, T]; L^p_{\omega_\delta}(\mathbb{R}^3))$, which gives $(\mathbf{u}\otimes \mathbf{u}, \mathbf{B}\otimes \mathbf{B}, \frac{|\mathbf{B}|^2}{2})\in L^\frac{p}{2}([0, T]; L^\frac{p}{2}_{\omega_\delta}(\mathbb{R}^3))$. Thus, by point ${\rm{(iii)}}$ of Lemma \ref{Lem-1}, one has $P \in L^\frac{p}{2}([0, T]; L^\frac{p}{2}_{\omega_\delta}(\mathbb{R}^3))$ for $\frac{9}{4}< \delta< 3$.

Next, we will show that $\nabla \Pi - \nabla P=0$, where $\Pi$ is the pressure term in \eqref{THMHD}. Let $\varepsilon >0$ (small enough) and $\psi(x) \in C^\infty_0 (\mathbb{R}^3)$. Let $\alpha(t) \in C^\infty_0 (\mathbb{R})$ be
a function such that for $|t|>\varepsilon$, $\alpha(t)=0$. We see that for all $t\in (\varepsilon, T-\varepsilon)$,
$(\alpha\psi)\ast\nabla \Pi \in \mathcal{D}'((\varepsilon, T-\varepsilon)\times\mathbb{R}^3)$ and $(\alpha\psi)\ast\nabla P \in \mathcal{D}'((\varepsilon, T-\varepsilon)\times\mathbb{R}^3)$.
Then, we define $A_\varepsilon (t)$ given by
\begin{equation*}
  A_\varepsilon (t)=(\alpha\psi)\ast\nabla \Pi (\cdot, t)- (\alpha\psi)\ast\nabla P (\cdot, t).
\end{equation*}
We claim that $A_\varepsilon (t)\in \mathcal{S}'(\mathbb{R}^3)$. Indeed, we get by using the vector identity $\operatorname{curl}\mathbf{B}\times\mathbf{B}= -\nabla \frac{|\mathbf{B}|^2}{2}+(\mathbf{B}\cdot\nabla) \mathbf{B}$ and the system \eqref{THMHD} that
\begin{equation*}
  \nabla \Pi =-\partial_t \mathbf{u} + \Delta \mathbf{u} - \operatorname{div}(\mathbf{u}\otimes \mathbf{u})- \operatorname{div}(\mathbf{B}\otimes \mathbf{B}- \frac{1}{2}|\mathbf{B}|^2I).
\end{equation*}
Here $I$ is the identical matrix.
Thus,
\begin{equation}\label{A}
     \begin{split}
  A_\varepsilon (t)&=\big[\big(-(\partial_t \alpha)\psi +\alpha\Delta\psi\big)\ast \mathbf{u}\big](\cdot, t)-\big[(\alpha\ast \nabla \psi)\ast (\mathbf{u}\otimes \mathbf{u})\big](\cdot, t)\\
  &\quad -\big[(\alpha\ast \nabla \psi)\ast (\mathbf{B}\otimes \mathbf{B}- \frac{1}{2}|\mathbf{B}|^2I)\big](\cdot, t) -\big[(\alpha \nabla\psi)\ast P\big](\cdot, t).
     \end{split}
\end{equation}

We successively prove that each term on the right hand side of \eqref{A} belongs to the space $L^\frac{p}{2}_{\omega_\delta}(\mathbb{R}^3)$ for $\frac{9}{4}< \delta< 3$.
For the first term of the right in \eqref{A}, since for $\psi \in C^\infty_0 (\mathbb{R}^3)$ and a test function $f$, we have the pointwise estimate $|(\psi\ast f)(x)|\leq C(\varphi) \mathcal{M}_f (x)$, where $\mathcal{M}$ denotes the Hardy-Littlewood maximal function operator. Then, using point ${\rm{(iii)}}$ of Lemma \ref{Lem-1} implies convolution with test functions is a bounded operator on $L^p_{\omega_\delta}(\mathbb{R}^3)$. Noting that
$\mathbf{u} \in L^p([0, T]; L^p_{\omega_\delta}(\mathbb{R}^3))$, we get $\big[\big(-(\partial_t \alpha)\psi +\alpha\Delta \psi\big)\ast \mathbf{u}\big](\cdot, t)\in L^p_{\omega_\delta}(\mathbb{R}^3)$, which together with
the embedding \eqref{em} gives
$\big[\big(-(\partial_t \alpha)\psi +\alpha\Delta \psi\big)\ast \mathbf{u}\big](\cdot, t)\in L^\frac{p}{2}_{\omega_\delta}(\mathbb{R}^3)$ for $\frac{9}{4}< \delta< 3$.

For the second and third terms of the right in \eqref{A}, since $\mathbf{u}\otimes \mathbf{u}\in L^\frac{p}{2}([0, T]; L^\frac{p}{2}_{\omega_\delta}(\mathbb{R}^3))$
and $(\mathbf{B}\otimes \mathbf{B}, \frac{|\mathbf{B}|^2}{2})\in L^\frac{p}{2}([0, T]; L^\frac{p}{2}_{\omega_\delta}(\mathbb{R}^3))$,
we have $\big[(\alpha\ast \nabla \psi)\ast (\mathbf{u}\otimes \mathbf{u})\big](\cdot, t)\in L^\frac{p}{2}_{\omega_\delta}(\mathbb{R}^3)$ and
$\big[(\alpha\ast \nabla \psi)\ast (\mathbf{B}\otimes \mathbf{B}- \frac{1}{2}|\mathbf{B}|^2I)\big](\cdot, t)\in L^\frac{p}{2}_{\omega_\delta}(\mathbb{R}^3)$. Here we have used the fact convolution with test functions is
a bounded operator on the space $L^\frac{p}{2}_{\omega_\delta}(\mathbb{R}^3)$.

For the fourth term of the right in \eqref{A}, noting $P\in L^\frac{p}{2}([0, T]; L^\frac{p}{2}_{\omega_\delta}(\mathbb{R}^3))$, we have
$\big[(\alpha \nabla\psi)\ast P\big](\cdot, t)\in L^\frac{p}{2}([0, T]; L^\frac{p}{2}_{\omega_\delta}(\mathbb{R}^3))$.

Hence, we have $A_\varepsilon (t)\in L^\frac{p}{2}_{\omega_\delta}(\mathbb{R}^3)$, and then  $A_\varepsilon (t)\in \mathcal{S}'(\mathbb{R}^3)$. In addition, taking the divergence of the first equation in
\eqref{THMHD}, we obtain from $\operatorname{div} \partial_t \mathbf{u}= \operatorname{div} \Delta \mathbf{u}=0$ that
$$\Delta \Pi=-\sum_{i,j=1}^3\partial_i\partial_j(u_iu_j-B_iB_j)- \Delta\frac{|\mathbf{B}|^2}{2},$$
which implies $\Delta(\Pi- P)=0$. Then we have $\Delta A_\varepsilon (t)=0$ and $A_\varepsilon (t)$ is a polynomial from $A_\varepsilon (t)\in \mathcal{S}'(\mathbb{R}^3)$. Noticing
$A_\varepsilon (t)\in L^\frac{p}{2}_{\omega_\delta}(\mathbb{R}^3)$, one has $A_\varepsilon (t)=0$. Moreover, we use the approximation of the identity
$\frac{1}{\varepsilon^4}\alpha\left(\frac{t}{\varepsilon}\right)\psi\left(\frac{x}{\varepsilon}\right)$ to write
\begin{equation*}
     \begin{split}
  \nabla (\Pi- P)(\cdot, t)=\lim_{\varepsilon\rightarrow 0}A_\varepsilon (t)=0.
     \end{split}
\end{equation*}
The proof is completed.
\end{proof}


\renewcommand{\theequation}{\thesection.\arabic{equation}}
\setcounter{equation}{0} 

\section{Caccioppoli type inequalities}
In this section, we derive the Caccioppoli type inequalities, which are the key to our proof.
Let us first give the estimate in the framework of Lebesgue spaces.
\begin{prop}\label{prop-1}
Let $(\mathbf{u}, \mathbf{B})$ be a smooth solution to \eqref{HMHD}. For $3\leq p <+\infty$, if $(\mathbf{u}, \mathbf{B})\in {L^p_{loc}(\mathbb{R}^3)}$ and $\nabla \mathbf{B}\in {L^2(\mathbb{R}^3)}$,
then we have
\begin{equation}\label{p1}
     \begin{split}
     &\int_{B_R}|\nabla \mathbf{u}|^2 \, dx + \int_{B_R}|\operatorname{curl} \mathbf{B}|^2 \, dx \\
     &\leq CR^{1-\frac{6}{p}}\|\mathbf{u}\|^2_{L^p(C_R)}+ C R^{1-\frac{6}{p}}\|\mathbf{B}\|^2_{L^p(C_R)}+ C R^{2-\frac{9}{p}}\|\mathbf{u}\|^3_{L^p(C_R)}+ C R^{-\frac{1}{2}}\|\nabla \mathbf{B}\|^3_{L^2(C_R)}\\
     &\quad  +C R^{2-\frac{9}{p}}\|\mathbf{u}\|_{L^p(C_R)}\|\mathbf{B}\|^2_{L^p(C_R)} + C R^{2-\frac{9}{p}}\|\mathbf{u}\|_{L^p(C_R)}\|\Pi\|^2_{L^\frac{p}{2}(C_R)},
     \end{split}
\end{equation}
where 
the constant $C$ is independent of $R\geq 1$.
\end{prop}

\begin{proof}
First, we introduce the following cut-off function: let $\varphi \in C_0^\infty (\mathbb{R}^3)$ be a nonincreasing radial function such that
\begin{equation*}
\varphi(|x|)=
    \begin{cases}
        1, \quad \mathrm{if}\,  |x|< 1, \\
        [0, 1], \quad \mathrm{if}\,  1\leq|x|\leq 2,\\
        0, \quad \mathrm{if}\,  |x|> 2.
    \end{cases}
\end{equation*}
Define $\varphi_R(x)=\varphi(\frac{|x|}{R})$  for each given $R$. Then the function $\varphi_R(x)$ satisfies
\begin{equation*}
     \begin{split}
   \|\nabla^k \varphi_R \|_{L^\infty} \leq C R^{-k}
     \end{split}
\end{equation*}
for $k=0,1,2$ with some positive constant $C$.

Taking the $L^2$- inner product of the $u$ equation of $\eqref{HMHD}$ with $\mathbf{u}\varphi_R^2$ and integrating by parts over $\mathbb{R}^3$, we obtain
\begin{equation}\label{3-1}
     \begin{split}
   \int_{\mathbb{R}^3}|\nabla \mathbf{u}|^2 \varphi_R^2\, dx &= -2\int_{\mathbb{R}^3} \varphi_R\nabla \mathbf{u} : (\mathbf{u}\otimes \nabla \varphi_R) \, dx
   -\int_{\mathbb{R}^3}(\mathbf{u}\cdot \nabla) \mathbf{u}\cdot \mathbf{u}\varphi_R^2 \, dx\\
  & \quad - \int_{\mathbb{R}^3}\nabla \Pi\cdot \mathbf{u}\varphi_R^2\, dx
   +\int_{\mathbb{R}^3} \operatorname{curl} \mathbf{B}\times \mathbf{B}\cdot \mathbf{u}\varphi_R^2 \, dx.
     \end{split}
\end{equation}
Similarly, taking the $L^2$- inner product of the $B$ equation of $\eqref{HMHD}$ with $\mathbf{B}\varphi_R^2$ and integrating by parts over $\mathbb{R}^3$, we have
\begin{equation}\label{3-2}
     \begin{split}
   &\int_{\mathbb{R}^3}|\operatorname{curl} \mathbf{B}|^2 \varphi_R^2 \, dx = -2\int_{\mathbb{R}^3} \varphi_R \operatorname{curl}\mathbf{B}\cdot (\nabla\varphi_R \times \mathbf{B})\, dx
   -\int_{\mathbb{R}^3} (\mathbf{B}\times \mathbf{u})\cdot (\varphi_R^2\operatorname{curl} \mathbf{B})\, dx \\
   &\qquad -2\int_{\mathbb{R}^3}\varphi_R(\mathbf{B}\times \mathbf{u})\cdot( \nabla\varphi_R \times\mathbf{B})\, dx- 2\int_{\mathbb{R}^3}\varphi_R(\operatorname{curl} \mathbf{B}\times\mathbf{B})\cdot ( \nabla\varphi_R \times\mathbf{B}) \, dx,
     \end{split}
\end{equation}
where we have used the fact $\Delta \mathbf{B}= \nabla \operatorname{div}\mathbf{B}-\operatorname{curl}(\operatorname{curl}\mathbf{B}).$

Combining \eqref{3-1} and \eqref{3-2},  it gives
\begin{equation}\label{3-3}
     \begin{split}
   &\int_{\mathbb{R}^3}|\nabla \mathbf{u}|^2 \varphi_R^2\, dx + \int_{\mathbb{R}^3}|\operatorname{curl} \mathbf{B}|^2 \varphi_R^2 \, dx \\
   &= -2 \int_{\mathbb{R}^3} \varphi_R \nabla \mathbf{u} : (\mathbf{u}\otimes \nabla \varphi_R) \, dx
   -\int_{\mathbb{R}^3}(\mathbf{u}\cdot \nabla) \mathbf{u}\cdot \mathbf{u}\varphi_R^2 \, dx\\
   &\quad -2\int_{\mathbb{R}^3} \varphi_R \operatorname{curl}\mathbf{B}\cdot(\nabla\varphi_R \times \mathbf{B}) \, dx
   -2 \int_{\mathbb{R}^3} \varphi_R(\mathbf{B}\times \mathbf{u})\cdot( \nabla\varphi_R \times\mathbf{B})\, dx\\
   &\quad -2\int_{\mathbb{R}^3} \varphi_R(\operatorname{curl} \mathbf{B}\times\mathbf{B})\cdot ( \nabla\varphi_R\times\mathbf{B}) \, dx- \int_{\mathbb{R}^3}\nabla \Pi\cdot \mathbf{u}\varphi_R^2\, dx\\
   &=\sum_{i=1}^6 I_i.
     \end{split}
\end{equation}

In what follows, we estimate each term on the right-hand side of \eqref{3-3} separately. For $I_1$, it can be deduced from the integration by parts that
\begin{equation*}
     \begin{split}
   I_1
   &=\int_{\mathbb{R}^3} \mathbf{u}\cdot\operatorname{div}(\mathbf{u} \otimes \nabla (\varphi_R^2)) \, dx\\
   &=\int_{\mathbb{R}^3} \mathbf{u}\cdot (\nabla \mathbf{u} \cdot \nabla (\varphi_R^2)+\mathbf{u} \Delta (\varphi_R^2)) \, dx\\
   &=\frac{1}{2}\int_{\mathbb{R}^3}|\mathbf{u}|^2 \Delta (\varphi_R^2) \, dx\\
    &=\int_{\mathbb{R}^3}|\mathbf{u}|^2 (\varphi_R\Delta \varphi_R+|\nabla \varphi_R|^2) \, dx.
     \end{split}
\end{equation*}
Applying the H\"{o}lder inequality implies
\begin{equation*}\label{3-4}
     \begin{split}
   |I_1|&\leq \int_{C_R}|\mathbf{u}|^2 (|\varphi_R\Delta \varphi_R|+|\nabla \varphi_R|^2) \, dx\\
   &\leq C  R^{-2}\||\mathbf{u}|^2\|_{L^{\frac{p}{2}}(C_R)}\|1\|_{L^{\frac{p}{p-2}}(C_R)}\\
   &\leq C  R^{1-\frac{6}{p}}\|\mathbf{u}\|^2_{L^p(C_R)}.
     \end{split}
\end{equation*}
For $I_2$, using integration by parts, we obtain from $\operatorname{div} \mathbf{u}=0$ that
\begin{equation*}
     \begin{split}
   I_2&=-\int_{\mathbb{R}^3} (\mathbf{u}\cdot \nabla) \mathbf{u}\cdot (\mathbf{u}\varphi_R^2) \, dx
   =-\frac{1}{2}\int_{\mathbb{R}^3} \mathbf{u}\cdot \nabla |\mathbf{u}|^2 \varphi_R^2  \, dx\\
   &=\frac{1}{2}\int_{\mathbb{R}^3} |\mathbf{u}|^2\operatorname{div}(\mathbf{u}\varphi_R^2)  \, dx
  =\int_{\mathbb{R}^3}|\mathbf{u}|^2\mathbf{u}\cdot (\varphi_R\nabla \varphi_R)  \, dx.
     \end{split}
\end{equation*}
Thus,
\begin{equation*}\label{3-5}
     \begin{split}
   |I_2|&\leq  \int_{C_R}|\mathbf{u}|^2|\mathbf{u}||\varphi_R\nabla \varphi_R| \, dx\\
   &\leq CR^{-1}\||\mathbf{u}|^2\|_{L^{\frac{p}{2}}(C_R)}\|\mathbf{u}\|_{L^{p}(C_R)}\|1\|_{L^{\frac{p}{p-3}}(C_R)}\\
   &\leq CR^{2-\frac{9}{p}}\|\mathbf{u}\|_{L^p(C_R)}^3.
     \end{split}
\end{equation*}
For $I_3$, using Young's inequality, we get
\begin{equation*}\label{3-6}
     \begin{split}
   |I_3|&\leq \int_{C_R} |\varphi_R\operatorname{curl}\mathbf{B}||\mathbf{B}||\nabla \varphi_R| \, dx\\
   &\leq CR^{-1}\|\varphi_R\operatorname{curl}\mathbf{B}\|_{L^2(C_R)}\|\mathbf{B}\|_{L^p(C_R)}\|1\|_{L^\frac{2p}{p-2}(C_R)}\\
   &\leq CR^{\frac{1}{2}-\frac{3}{p}}\|\varphi_R\operatorname{curl}\mathbf{B}\|_{L^2(C_R)}\|\mathbf{B}\|_{L^p(C_R)}\\
   &\leq \frac{1}{2}\|\varphi_R\operatorname{curl}\mathbf{B}\|_{L^2(C_R)}^2+CR^{1-\frac{6}{p}}\|\mathbf{B}\|^2_{L^p(C_R)}.
     \end{split}
\end{equation*}
Again, applying the H\"{o}lder inequality gives
\begin{equation*}\label{3-7}
     \begin{split}
   |I_4|&\leq CR^{-1}\int_{C_R} |\mathbf{B}|^2|\mathbf{u}| \, dx\\
   &\leq CR^{-1}\||\mathbf{B}|^2\|_{L^\frac{p}{2}(C_R)}\|\mathbf{u}\|_{L^p(C_R)}\|1\|_{L^\frac{p}{p-3}(C_R)}\\
   &\leq CR^{2-\frac{9}{p}}\|\mathbf{B}\|_{L^p(C_R)}^2\|\mathbf{u}\|_{L^p(C_R)}.
     \end{split}
\end{equation*}
For $I_5$, using the embedding $\dot{H}^1(\mathbb{R}^3)\hookrightarrow L^6(\mathbb{R}^3)$ implies
\begin{equation}\label{3-8}
     \begin{split}
   |I_5|&\leq CR^{-1}\int_{C_R} |\varphi_R\operatorname{curl}\mathbf{B}||\mathbf{B}|^2 \, dx\\
   &\leq CR^{-1}\|\operatorname{curl}\mathbf{B}\|_{L^2(C_R)}\||\mathbf{B}|^2\|_{L^3(C_R)}\|1\|_{L^6(C_R)}\\
   &\leq CR^{-\frac{1}{2}}\|\nabla\mathbf{B}\|_{L^2(C_R)}\|\mathbf{B}\|_{L^6(C_R)}^2\\
   &\leq CR^{-\frac{1}{2}}\|\nabla\mathbf{B}\|_{L^2}^3.
     \end{split}
\end{equation}
Finally, for $I_6$, using integration by parts and the fact $\operatorname{div} \mathbf{u}=0$, we  get
\begin{equation*}
     \begin{split}
   I_6
   &=\int_{\mathbb{R}^3}\Pi\operatorname{div}(\mathbf{u} \varphi_R^2) \, dx\\
   &=\int_{\mathbb{R}^3}\Pi\operatorname{div}\mathbf{u} \varphi_R^2\, dx+2\int_{\mathbb{R}^3}\varphi_R\Pi\mathbf{u}\cdot\nabla\varphi_R\, dx\\
   &=2\int_{\mathbb{R}^3} \varphi_R\Pi\mathbf{u}\cdot\nabla\varphi_R\, dx.
     \end{split}
\end{equation*}
Then,
\begin{equation*}\label{3-9}
     \begin{split}
   |I_6|&\leq CR^{-1}\int_{C_R} |\Pi||\mathbf{u}|\, dx\\
   &\leq CR^{-1}\|\Pi\|_{L^\frac{p}{2}(C_R)}\|\mathbf{u}\|_{L^p(C_R)}\|1\|_{L^\frac{p}{p-3}(C_R)}\\
   &\leq CR^{2-\frac{9}{p}}\|\Pi\|_{L^\frac{p}{2}(C_R)}\|\mathbf{u}\|_{L^p(C_R)}.\\
     \end{split}
\end{equation*}

Consequently, substituting the estimates $I_i$ for $i=1,2, ... ,6$ into \eqref{3-3}, we deduce that
\begin{equation*}\label{3-10}
     \begin{split}
   &\int_{B_R}|\nabla \mathbf{u}|^2\, dx + \frac{1}{2}\int_{B_R}|\operatorname{curl} \mathbf{B}|^2 \, dx\\
   &\leq C R^{1-\frac{6}{p}}\|\mathbf{u}\|^2_{L^p(C_R)}+ CR^{1-\frac{6}{p}}\|\mathbf{B}\|^2_{L^p(C_R)}+ CR^{2-\frac{9}{p}}\|\mathbf{u}\|^3_{L^p(C_R)}+CR^{2-\frac{9}{p}}\|\mathbf{u}\|_{L^p(C_R)}\|\mathbf{B}\|^2_{L^p(C_R)}\\
   &\qquad + CR^{-\frac{1}{2}}\|\nabla \mathbf{B}\|^3_{L^2(C_R)}
   +CR^{2-\frac{9}{p}}\|\mathbf{u}\|_{L^p(C_R)}\|\Pi\|_{L^\frac{p}{2}(C_R)}
     \end{split}
\end{equation*}
for all $R\geq 1$ and the constant $C$ independent of $R$. This completes the proof of Proposition \ref{prop-1}.
\end{proof}

Next, in the framework of local Morrey spaces, the Caccioppoli type estimate is stated as follows:
\begin{prop}\label{prop-2}
Let $(\mathbf{u}, \mathbf{B})$ be a smooth solution to \eqref{HMHD}. For $0< \gamma< 3 \leq p <+\infty$, if $(\mathbf{u}, \mathbf{B})\in {M^p_\gamma (\mathbb{R}^3)}$ and $\nabla \mathbf{B}\in {L^2(\mathbb{R}^3)}$,
then we have
\begin{equation}\label{p2}
     \begin{split}
     &\int_{B_R}|\nabla \mathbf{u}|^2 \, dx + \int_{B_R}|\operatorname{curl} \mathbf{B}|^2 \, dx \\
     &\leq C R^{3\eta}(R^{-\gamma}\int_{C_R}|\mathbf{u}|^p \,dx)^{\frac{3}{p}}+ C R^{-\frac{1}{3}}R^{2\eta}(R^{-\gamma}\int_{C_R}|\mathbf{u}|^p \,dx)^{\frac{2}{p}}
     + C R^{-\frac{1}{2}}\|\nabla \mathbf{B}\|^3_{L^2(C_R)}\\
     &\quad +C R^{3\eta}(R^{-\gamma}\int_{C_R}|\mathbf{B}|^p \,dx)^{\frac{2}{p}}(R^{-\gamma}\int_{C_R}|\mathbf{u}|^p \,dx)^{\frac{1}{p}}+C R^{-\frac{1}{3}}R^{2\eta}(R^{-\gamma}\int_{C_R}|\mathbf{B}|^p \,dx)^{\frac{2}{p}}\\
     &\quad +
      C R^{3\eta}(R^{-\gamma}\int_{C_R}|\mathbf{u}|^p \,dx)^{\frac{1}{p}}(\|\mathbf{u}\|^2_{M^p_\gamma(\mathbb{R}^3)}+\|\mathbf{B}\|^2_{M^p_\gamma(\mathbb{R}^3)}),
     \end{split}
\end{equation}
where 
the constant $C$ is independent of $R\geq 1$.
\end{prop}

\begin{proof}
Using Proposition \ref{prop-1a}, we know that $\nabla \Pi =\nabla P$, where $P$ is given in \eqref{P}. Thus, we consider the system
\eqref{HMHD} with the term $\nabla P$ instead of the term $\nabla \Pi$. Since $(\mathbf{u}, \mathbf{B})\in {M^p_\gamma (\mathbb{R}^3)}$, one has $(\mathbf{u}, \mathbf{B})\in {L^p_{loc}(\mathbb{R}^3)}$.
Then we obtain from Proposition \ref{prop-1} that for $3 \leq p <+\infty$,
\begin{equation}\label{p11}
     \begin{split}
     &\int_{B_R}|\nabla \mathbf{u}|^2 \, dx + \int_{B_R}|\operatorname{curl} \mathbf{B}|^2 \, dx \\
     &\leq CR^{1-\frac{6}{p}}\|\mathbf{u}\|^2_{L^p(C_R)}+ C R^{1-\frac{6}{p}}\|\mathbf{B}\|^2_{L^p(C_R)}+ C R^{2-\frac{9}{p}}\|\mathbf{u}\|^3_{L^p(C_R)}+ C R^{-\frac{1}{2}}\|\nabla \mathbf{B}\|^3_{L^2(C_R)}\\
     &\quad  +C R^{2-\frac{9}{p}}\|\mathbf{u}\|_{L^p(C_R)}\|\mathbf{B}\|^2_{L^p(C_R)} + C R^{2-\frac{9}{p}}\|\mathbf{u}\|_{L^p(C_R)}\| P\|^2_{L^\frac{p}{2}(C_R)}.
     \end{split}
\end{equation}

A direct computation gives
\begin{equation*}
     \begin{split}
   R^{2-\frac{9}{p}}\|\mathbf{u}\|_{L^p(C_R)}^3&=R^{-\frac{2}{p}\gamma}R^{\frac{2}{p}\gamma+2-\frac{9}{p}}(\int_{C_R}|\mathbf{u}|^p \,dx)^{\frac{2}{p}}(\int_{C_R}|\mathbf{u}|^p \,dx)^{\frac{1}{p}}\\
   &=R^{\frac{3}{p}\gamma+2-\frac{9}{p}}(R^{-\gamma}\int_{C_R}|\mathbf{u}|^p \,dx)^{\frac{2}{p}}(R^{-\gamma}\int_{C_R}|\mathbf{u}|^p \,dx)^{\frac{1}{p}}.
     \end{split}
\end{equation*}
Similarly, we obtain
\begin{equation*}
     \begin{split}
   &R^{1-\frac{6}{p}}\|\mathbf{u}\|_{L^p(C_R)}^2 + R^{1-\frac{6}{p}}\|\mathbf{B}\|_{L^p(C_R)}^2 \\
   &=R^{-\frac{1}{3}-\frac{2}{p}\gamma}R^{\frac{2}{p}\gamma+\frac{4}{3}-\frac{6}{p}}(\int_{C_R}|\mathbf{u}|^p \,dx)^{\frac{2}{p}}
   +R^{-\frac{1}{3}-\frac{2}{p}\gamma}R^{\frac{2}{p}\gamma+\frac{4}{3}-\frac{6}{p}}(\int_{C_R}|\mathbf{B}|^p \,dx)^{\frac{2}{p}}\\
   &=R^{-\frac{1}{3}}R^{\frac{2}{p}\gamma+\frac{4}{3}-\frac{6}{p}}(R^{-\gamma}\int_{C_R}|\mathbf{u}|^p \,dx)^{\frac{2}{p}}+ R^{-\frac{1}{3}}R^{\frac{2}{p}\gamma+\frac{4}{3}-\frac{6}{p}}(R^{-\gamma}\int_{C_R}|\mathbf{B}|^p \,dx)^{\frac{2}{p}}
     \end{split}
\end{equation*}
and
\begin{equation*}
     \begin{split}
   R^{2-\frac{9}{p}}\|\mathbf{B}\|_{L^p(C_R)}^2\|\mathbf{u}\|_{L^p(C_R)}&=R^{\frac{2}{p}\gamma+2-\frac{9}{p}}(R^{-\gamma}\int_{C_R}|\mathbf{B}|^p \,dx)^{\frac{2}{p}}(\int_{C_R}|\mathbf{u}|^p \,dx)^{\frac{1}{p}}\\
   &=R^{\frac{3}{p}\gamma+2-\frac{9}{p}}(R^{-\gamma}\int_{C_R}|\mathbf{B}|^p \,dx)^{\frac{2}{p}}(R^{-\gamma}\int_{C_R}|\mathbf{u}|^p \,dx)^{\frac{1}{p}}.
     \end{split}
\end{equation*}

Finally, we are in a position to estimate $R^{2-\frac{9}{p}}\|\mathbf{u}\|_{L^p(C_R)}\|P\|^2_{L^\frac{p}{2}(C_R)}$.
Noticing that the term $P$ is defined through $\mathbf{u}$ and $\mathbf{B}$ in \eqref{P}, we deduce from point ${\rm{(i)}}$ of Lemma \ref{Lem-1} that  for all $0<\gamma<3\leq p<+\infty$,
\begin{equation*}
     \begin{split}
   \|P\|_{M^\frac{p}{2}_\gamma(\mathbb{R}^3)}\leq C(p, \gamma)(\|\mathbf{u}\|^2_{M^p_\gamma(\mathbb{R}^3)}+\|\mathbf{B}\|^2_{M^p_\gamma(\mathbb{R}^3)}).
     \end{split}
\end{equation*}
Moreover, recalling the definition of $\|\cdot\|_{M^p_\gamma(\mathbb{R}^3)}$, see \eqref{def-2}, we get
\begin{equation*}
     \begin{split}
   R^{2-\frac{9}{p}}\| P\|_{L^\frac{p}{2}(C_R)}\|\mathbf{u}\|_{L^p(C_R)}&=R^{\frac{3}{p}\gamma+2-\frac{9}{p}}(R^{-\gamma}\int_{C_R}|\mathbf{u}|^p \,dx)^{\frac{1}{p}}(R^{-\gamma}\int_{C_R}| P|^\frac{p}{2} \,dx)^{\frac{2}{p}}\\
   &\leq R^{\frac{3}{p}\gamma+2-\frac{9}{p}}(R^{-\gamma}\int_{C_R}|\mathbf{u}|^p \,dx)^{\frac{1}{p}} \| P\|_{M^\frac{p}{2}_\gamma(\mathbb{R}^3)}\\
   &\leq CR^{\frac{3}{p}\gamma+2-\frac{9}{p}}(R^{-\gamma}\int_{C_R}|\mathbf{u}|^p \,dx)^{\frac{1}{p}}(\|\mathbf{u}\|^2_{M^p_\gamma(\mathbb{R}^3)}+\|\mathbf{B}\|^2_{M^p_\gamma(\mathbb{R}^3)}).
     \end{split}
\end{equation*}

Inserting all the above estimates into \eqref{p11}, we complete the proof of Proposition \ref{prop-2}.
\end{proof}


\renewcommand{\theequation}{\thesection.\arabic{equation}}
\setcounter{equation}{0} 

\section{The Proofs of Theorems}

With the help of Proposition \ref{prop-2}, we can now complete the proofs of Theorems 1 and 2.
For simplicity, we denote
\begin{equation*}
     \begin{split}
D_1:&= R^{3\eta}(R^{-\gamma}\int_{C_R}|\mathbf{u}|^p \,dx)^{\frac{3}{p}}+ R^{-\frac{1}{3}}R^{2\eta}(R^{-\gamma}\int_{C_R}|\mathbf{u}|^p \,dx)^{\frac{2}{p}}
     +  R^{-\frac{1}{2}}\|\nabla \mathbf{B}\|^3_{L^2(C_R)}\\
    &+ R^{3\eta}(R^{-\gamma}\int_{C_R}|\mathbf{B}|^p \,dx)^{\frac{2}{p}}(R^{-\gamma}\int_{C_R}|\mathbf{u}|^p \,dx)^{\frac{1}{p}}+ R^{-\frac{1}{3}}R^{2\eta}(R^{-\gamma}\int_{C_R}|\mathbf{B}|^p \,dx)^{\frac{2}{p}}\\
    &+ R^{3\eta}(R^{-\gamma}\int_{C_R}|\mathbf{u}|^p \,dx)^{\frac{1}{p}}(\|\mathbf{u}\|^2_{M^p_\gamma(\mathbb{R}^3)}+\|\mathbf{B}\|^2_{M^p_\gamma(\mathbb{R}^3)}).
     \end{split}
\end{equation*}
Then the inequality \eqref{p2} writes the following
\begin{equation}\label{4-1}
     \begin{split}
     \int_{B_R}|\nabla \mathbf{u}|^2 \, dx + \int_{B_R}|\operatorname{curl} \mathbf{B}|^2 \, dx \leq C D_1.
     \end{split}
\end{equation}

\begin{proof}[Proof of Theorem 1]
Assume that $\nabla \mathbf{B}\in L^2(\mathbb{R}^3)$ and $(\mathbf{u}, \mathbf{B})\in M^p_\gamma(\mathbb{R}^3)$ with $0<\gamma<3\leq p<+\infty$.
We divide our proof into three cases.

\textbf{(i). The case when $\eta(p, \gamma)< 0$}.
According to the definition of $\|\cdot\|_{M^p_\gamma(\mathbb{R}^3)}$, we obtain
\begin{equation}\label{4-2}
     \begin{split}
  D_1\leq R^{3\eta}\left(\|\mathbf{u}\|_{M^p_\gamma}^2+ \|\mathbf{B}\|_{M^p_\gamma}^2\right)\|\mathbf{u}\|_{M^p_\gamma}
  +R^{-\frac{1}{3}}R^{2\eta}\left(\|\mathbf{u}\|_{M^p_\gamma}^2+ \|\mathbf{B}\|_{M^p_\gamma}^2\right)+R^{-\frac{1}{2}}\|\nabla \mathbf{B}\|_{L^2}^2.
     \end{split}
\end{equation}
Passing $R\rightarrow +\infty$ in \eqref{4-1}, we see
\begin{equation*}
     \begin{split}
  \lim_{R\rightarrow +\infty} D_1
  =0.
     \end{split}
\end{equation*}
Thus, it leads to
\begin{equation*}
     \begin{split}
     \lim_{R\rightarrow +\infty}(\int_{B_R}|\nabla \mathbf{u}|^2 \, dx + \int_{B_R}|\operatorname{curl} \mathbf{B}|^2 \, dx ) =0,
     \end{split}
\end{equation*}
which together with Levi's monotone convergence theorem yields
\begin{equation*}
     \begin{split}
     \int_{\mathbb{R}^3}|\nabla \mathbf{u}|^2 \, dx + \int_{\mathbb{R}^3}|\operatorname{curl} \mathbf{B}|^2 \, dx  =0.
     \end{split}
\end{equation*}

Using the fact that
$$\|\nabla \mathbf{B}\|_{L^2(\mathbb{R}^3)}\leq C\left(\|\operatorname{div} \mathbf{B}\|_{L^2(\mathbb{R}^3)}+ \|\operatorname{curl} \mathbf{B}\|_{L^2(\mathbb{R}^3)}\right),$$
we conclude that $\|(\nabla \mathbf{u}, \nabla \mathbf{B})\|_{L^2(\mathbb{R}^3)}=0$ and $\mathbf{u}=\mathbf{B}= 0$ on $\mathbb{R}^3$ from $(\mathbf{u}, \mathbf{B})\in M^p_\gamma(\mathbb{R}^3).$

\textbf{(ii). The case when $\eta(p, \gamma)=0$}. Passing $R\rightarrow +\infty$ in \eqref{4-1} and by virtue of \eqref{4-2}, we get
\begin{equation*}
     \begin{split}
  \lim_{R\rightarrow +\infty} D_1
  \leq\|\mathbf{u}\|_{M^p_\gamma}^3+ \|\mathbf{B}\|_{M^p_\gamma}^2\|\mathbf{u}\|_{M^p_\gamma},
     \end{split}
\end{equation*}
and thus
\begin{equation*}
     \begin{split}
   \int_{\mathbb{R}^3}|\nabla \mathbf{u}|^2\, dx + \int_{\mathbb{R}^3}|\operatorname{curl} \mathbf{B}|^2 \, dx \leq C_0 \left(\|\mathbf{u}\|_{M^p_\gamma}^3+ \|\mathbf{B}\|_{M^p_\gamma}^2\|\mathbf{u}\|_{M^p_\gamma}\right).
     \end{split}
\end{equation*}
It follows from the assumption \eqref{con-0} that
\begin{equation*}
     \begin{split}
  \|(\nabla \mathbf{u}, \nabla \mathbf{B})\|_{L^2(\mathbb{R}^3)}^2\leq C_0\left(\|\mathbf{u}\|_{M^p_\gamma}^3+ \|\mathbf{B}\|_{M^p_\gamma}^2\|\mathbf{u}\|_{M^p_\gamma}\right)
  \leq C_0\delta \|(\nabla \mathbf{u}, \nabla \mathbf{B})\|_{L^2(\mathbb{R}^3)}^2 .
     \end{split}
\end{equation*}
Since $0< C_0\delta <1$, we conclude that $\|(\nabla \mathbf{u}, \nabla \mathbf{B})\|_{L^2(\mathbb{R}^3)}=0$ and  $\mathbf{u}=\mathbf{B}= 0$ on $\mathbb{R}^3$ .

\textbf{(iii). The case when $\eta(p, \gamma)>0$}. Here, we see $2\eta(p, \gamma)\leq 6\eta(p, \gamma)$ and $3\eta(p, \gamma)\leq 6\eta(p, \gamma)$. By direct calculation, we have
\begin{equation*}
     \begin{split}
R^{3\eta}(R^{-\gamma}\int_{C_R}|\mathbf{u}|^p \,dx)^{\frac{2}{p}}(R^{-\gamma}\int_{C_R}|\mathbf{u}|^p \,dx)^{\frac{1}{p}}
\leq \left[ R^{3\eta}(R^{-\gamma}\int_{C_R}|\mathbf{u}|^p \,dx)^{\frac{1}{p}} \right]^2 (R^{-\gamma}\int_{C_R}|\mathbf{u}|^p \,dx)^{\frac{1}{p}}
     \end{split}
\end{equation*}
and
\begin{equation*}
     \begin{split}
R^{-\frac{1}{3}}R^{2\eta}(R^{-\gamma}\int_{C_R}|\mathbf{u}|^p \,dx)^{\frac{2}{p}}\leq R^{-\frac{1}{3}}\left[ R^{3\eta}(R^{-\gamma}\int_{C_R}|\mathbf{u}|^p \,dx)^{\frac{1}{p}} \right]^2.
     \end{split}
\end{equation*}
Similarly, we get
\begin{equation*}
     \begin{split}
R^{3\eta}(R^{-\gamma}\int_{C_R}|\mathbf{B}|^p \,dx)^{\frac{2}{p}}(R^{-\gamma}\int_{C_R}|\mathbf{u}|^p \,dx)^{\frac{1}{p}}
\leq \left[R^{3\eta}(R^{-\gamma}\int_{C_R}|\mathbf{u}|^p \,dx)^{\frac{1}{p}}\right](R^{-\gamma}\int_{C_R}|\mathbf{B}|^p \,dx)^{\frac{2}{p}}
     \end{split}
\end{equation*}
and
\begin{equation*}
     \begin{split}
R^{-\frac{1}{3}}R^{2\eta}(R^{-\gamma}\int_{C_R}|\mathbf{B}|^p \,dx)^{\frac{2}{p}}\leq R^{-\frac{1}{3}}\left[ R^{3\eta}(R^{-\gamma}\int_{C_R}|\mathbf{B}|^p \,dx)^{\frac{1}{p}} \right]^2.
     \end{split}
\end{equation*}

Noting the condition \eqref{con-1}, passing $R\rightarrow +\infty$ in \eqref{4-1}, we deduce
\begin{equation*}
     \begin{split}
   \int_{\mathbb{R}^3}|\nabla \mathbf{u}|^2\, dx + \int_{\mathbb{R}^3}|\operatorname{curl} \mathbf{B}|^2 \, dx =0.
     \end{split}
\end{equation*}
Hence, we conclude that  $\mathbf{u}= \mathbf{B}= 0$ on $\mathbb{R}^3$.
\end{proof}

\begin{proof}[Proof of Theorem 2]
Assume that $\nabla \mathbf{B} \in L^2(\mathbb{R}^3)$, $\mathbf{u}\in M^p_{\gamma, 0}(\mathbb{R}^3)$ and $\mathbf{B}\in M^p_\gamma(\mathbb{R}^3)$ for $0<\gamma<3\leq p<+\infty$.
By the definition of $\|\cdot\|_{M^p_\gamma(\mathbb{R}^3)}$, we have
\begin{equation}\label{proof-2}
     \begin{split}
     D_1 &\leq R^{3\eta}(\|\mathbf{u}\|_{M^p_\gamma}^2+\|\mathbf{B}\|_{M^p_\gamma}^2)(R^{-\gamma}\int_{C_R}|\mathbf{u}|^p \,dx)^{\frac{1}{p}}+  R^{-\frac{1}{2}}\|\nabla \mathbf{B}\|^3_{L^2}\\
         &\quad + R^{-\frac{1}{3}}R^{2\eta}(\|\mathbf{u}\|_{M^p_\gamma}^2+\|\mathbf{B}\|_{M^p_\gamma}^2).
     \end{split}
\end{equation}
Considering the assumption $\eta(p, \gamma)= 0$,  letting $R\rightarrow +\infty$ in \eqref{4-1}, and applying Levi's monotone convergence theorem, we have
\begin{equation*}
     \begin{split}
   \int_{\mathbb{R}^3}|\nabla \mathbf{u}|^2\, dx + \int_{\mathbb{R}^3}|\operatorname{curl} \mathbf{B}|^2 \, dx =0.
     \end{split}
\end{equation*}
Hence, we get the desired identity $\mathbf{u}= \mathbf{B}= 0$ on $\mathbb{R}^3$.
\end{proof}

\noindent {\bf Acknowledgments.}
The work is partially supported by the National Natural Science Foundation of China under the grants 11571279
and 11601423.

\end{document}